\documentclass[12pt,a4paper]{amsart}
\usepackage{amssymb}
\usepackage{cite}
\usepackage{color}
\usepackage{graphicx}
\usepackage{setspace}

\newtheorem{theorem}{Theorem}
\theoremstyle{plain}

\newtheorem{corollary}{Corollary}

\newtheorem{lemma}{Lemma}

\newtheorem{proposition}{Proposition}

\numberwithin{equation}{section}

\input{tcilatex}

\begin{document}
\title[Convex and starshaped sets in $\mathcal{W}^{n}$]{Convex and
starshaped sets in manifolds without conjugate points}
\author{Sameh Shenawy}
\address{Basic Science Department, Modern Academy for Engineering and
Technology, Maadi, Egypt.}
\email{drssshenawy@eng.modern-academy.edu.eg, drshenawy@mail.com}

\begin{abstract}
Let $\mathcal{W}^{n}$ be the class of $C^{\infty }$ complete simply
connected $n-$dimensional manifolds without conjugate points. The hyperbolic
space as well as Euclidean space are good examples of such manifolds. Let $%
W\in \mathcal{W}^{n}$ and let $A$ be a subset of $W$. This article aims at
characterization and building convex and starshaped sets in this class from
inside. For example, it is proven that, for a compact starshaped set, the
convex kernel is the intersection of stars of extreme points only. Also, if
a closed unbounded convex set $A$ does not contain a totally geodesic
hypersurface and its boundary has no geodesic ray, then $A$ is the convex
hull of its extreme points. This result is a refinement of the well-known
Karein-Millman theorem.
\end{abstract}

\maketitle

\section{An Introduction}

Let $M$ be a $C^{\infty }$ complete Riemannian manifold. A vector field $J$
along a geodesic $\alpha $ is called a Jacobi vector field if 
\begin{equation*}
D_{T}^{2}J+\Re \left( \alpha ^{\prime },J\right) \alpha ^{\prime }=0
\end{equation*}%
$D$ is the covariant derivative and $\Re $ is the curvature tensor. Two
points on a geodesic $\alpha $ are said to be conjugate to each other if
there is a non-trivial Jacobi vector field along $\alpha $ that vanishes at
both of them. A geodesic $\alpha $ has no conjugate points if every Jacobi
field along $\alpha $ vanishes at most once. A $C^{\infty }$ complete
Riemannian manifold $M$ is called a manifold without conjugate points if
every geodesic of $M$ has no conjugate points. In this case, the exponential
map is a covering map at every point of $M$. Moreover, if $M$ is simply
connected, then $\exp _{p}$ is a diffeomorphism and $M$ has the property
that for every two distinct points $p$ and $q$ in $M$, there is a unique
geodesic joining them. Let $\mathcal{W}^{n}$ be the class of $C^{\infty }$
complete simply connected $n-$dimensional Riemannian manifolds without
conjugate points. The hyperbolic space $H^{n}$, the $n-$dimensional
Euclidean space $E^{n}$ and all manifolds with non-positive curvature are
good examples of such manifolds. We refer to \cite{Burns:1992,
Eberlein:1972, Emmerich:2013, Green:1954, Goto:1978,
Gulliver:1975,Ivanov:2014,Morse:1942, Ruggiero:2008, Shenawy:2011,
Sullivan:1974} and references therein for more details and examples of these
manifolds.

It is very nice to study the boundary of a closed set $A$ in $W\in \mathcal{W%
}^{n}$ and get a global properties of $A$. For instance, the Krein-Milman
theorem \cite{Lay:1982, Li:2006, Balashov:2002, GARCIA:1985, Oates:1971,
Valentine:1964} in the $n-$dimensional Euclidean space $E^{n}$ asserts that
every compact convex set is the convex hull of its extreme points i.e. given
a compact convex set $A\subset E^{n}$, one only need its extreme points $%
E\left( A\right) $ to recover the set shape.

The aim of this paper is to characterize convex and starshaped sets in
manifolds without conjugate points using their extreme points. Sufficient
conditions for a set $A$ in $W\in \mathcal{W}^{n}$ to be convex, totally
geodesic, and starshaped are considered. A generalization of Krein-Milman
theorem to the setting of closed unbounded convex sets is given. It is clear
that the convex kernel of a starshaped set $A\subset W$, $W\in \mathcal{W}%
^{2}$, is the intersection of stars of all points of $A$. In this work, it
is proven that, for a compact starshaped set, the convex kernel is the
intersection of stars of extreme points only. Moreover, the original
starshaped condition is replaced by a more general condition where the
intersection of the stars of certain extreme points is not empty. Thus we
get a characterization of starshaped sets in $\mathcal{W}^{2}$.

\section{Results}

Let $W\in \mathcal{W}^{n}$ and let $A$ be a non-empty subset of $W$. The
geodesic segment joining two points $p$ and $q$ is denoted by $\left[ pq%
\right] $. If $p$ is removed we write $\left( pq\right] $. The geodesic ray
with vertex at $p$ and passing through $q$ is denoted by $R\left( pq\right) $
while the geodesic passing through $p$ and $q$ is denoted by $G\left(
pq\right) $. We say that $p$ sees $q$ via $A$ if $\left[ pq\right] \subset A$%
. The set of all points of $A$ that $p$ sees via $A$ is called the star of $%
A $ at $p$ and is denoted by $A_{p}$. $A$ is a starshaped set if there is a
point $p\in A$ that sees every point in $A$ i.e. $A_{p}=A$. The set of all
such points $p$ is called the kernel of $A$ and is denoted by $\ker A$. $A$
is convex if $\ker A=A$. A point $p\in A$ is called an extreme point of $A$
if $p$ is not a relative interior point of any segment in $A$. The set of
all extreme points of $A$ is called the profile of $A$ and is denoted by $%
E\left( A\right) $. Note that the definition of extreme points is introduced
here to a non-convex set so it is somewhat different form the classical one.
The convex hull, $C\left( A\right) $, of $A$ is the intersection of all
convex subsets of $E^{n}$ that contain $A$. Three concepts of convex sets
were introduced to complete Riemannian manifolds in \cite{Alexander:1978}.
The three concepts coincide in complete simply connected Riemannian
manifolds without conjugate points since geodesics of these manifolds are
global minimizers.\cite{Beltagy:1993, BeltagyB:1994, BeltagyB:1988,
Jaume:2005,Shenawy:2011}.

We begin with the following important lemmas.

\begin{lemma}
\label{L1}Let $W\in \mathcal{W}^{n}$ and let $A$ be a closed subset of $W$.
If $a$ and $b$ are points of $A$ and $\left[ ab\right] \nsubseteq A$, then
there are two points $x,y\in \partial A\cap \left[ ab\right] $ such that $%
\left( xy\right) \cap A=\varphi $.
\end{lemma}

\begin{lemma}
\label{L2}Let $W\in \mathcal{W}^{n}$ and let $A$ be a compact subset of $W$.
Then $A$ has at least one extreme point.
\end{lemma}

\begin{proof}
Let $p$ be in $W\backslash A$. Define the real-valued continuous function $f$
on $A$ by $f\left( x\right) =d\left( p,x\right) ,x\in A$. Since $A$ is
compact, $f$ attains its maximum value at a point $y\in A$. Thus $A$ is a
subset of the closed disc $\overline{B}\left( p,r\right) $ with centre at $p$
and radius $r=d\left( p,y\right) $ defined by%
\begin{equation*}
\overline{B}\left( p,r\right) =\left\{ x\in W:d\left( p,x\right) \leq
r\right\}
\end{equation*}%
The point $y$ is an extreme point of $A$ since any geodesic segment
containing $p$ in its relative interior cuts the exterior of $A$.
\end{proof}

\begin{theorem}
Let $W\in \mathcal{W}^{2}$ and let $A$ be a compact starshaped subset of $W$%
. Then%
\begin{equation*}
\ker A=\bigcap\limits_{x\in E\left( A\right) }A_{x}
\end{equation*}
\end{theorem}

\begin{proof}
Let $B=\bigcap\limits_{x\in E\left( A\right) }A_{x}$. By the definition of
the kernel of a starshaped set we have%
\begin{equation*}
\ker A=\bigcap\limits_{x\in A}A_{x}\subset \bigcap\limits_{x\in E\left(
A\right) }A_{x}=B
\end{equation*}%
So, we need only to show that $B\subset A$. Let $x\in B\backslash \ker A$.
Then there is a point $y\in A$ such that $\left[ xy\right] \nsubseteq A$. By
Lemma \ref{L1}, we find two points $\overline{x}$, $\overline{y}$ in $%
\partial A\cap \left[ xy\right] $ such that $\left( \overline{x}\text{ }%
\overline{y}\right) \cap A=\phi $. Let $z\in \ker A,$ then $z$ sees $%
\overline{y}$ via $A$ and hence $R\left( z\overline{y}\right) \cap A$ is a
closed geodesic segment. Let $q\in \partial A$ such that $R\left( z\overline{%
y}\right) \cap A=\left[ zq\right] $. Suppose that $q\neq \overline{y}$.
Since $x\in E\left( A\right) $, $q$ sees $x$ via $A$. Then $z$ sees the
geodesic segment $\left[ xq\right] $ via $A$ and consequently $z$ sees $%
\left[ x\overline{y}\right] $ via $A$ which is a contradiction and $q$ is
not an extreme point i.e. there is a geodesic segment $\left[ ab\right]
\subset A$ such that $p\in \left( ab\right) $. It is clear from the choice
of $q$ that $\left( ab\right) \nsubseteq R\left( z\overline{y}\right) $.
Since $z\in \ker A$, $z$ sees $\left( ab\right) $ via $A$. Thus we get two
points $\overline{a}=\left[ za\right] \cap \left[ xy\right] $ and $\overline{%
b}=\left[ zb\right] \cap \left[ xy\right] $ such that $\overline{y}\in \left[
\overline{a}\overline{b}\right] $ which contradicts the choice of $\overline{%
y}$. So, $q=\overline{y}$. $\overline{y}$ is not an extreme point otherwise $%
\overline{y}$ sees $x$. Therefore, we get a geodesic segment $\left[ rs%
\right] $ such that $\overline{y}\in \left( rs\right) \subset A$. The
geodesic $G\left( rs\right) $ separates the points $x$ and $z$ otherwise, as
we do above, $z$ sees $\left( rs\right) $ via $A$ and we get a point $%
\overline{r}=\left[ zr\right] \cap \left[ xy\right] \in A$ that contradicts
the choice of $\overline{y}$. Let $H_{1}$ be the closed half space generated
by $G\left( x\overline{y}\right) $ that does not contain $z$ and let $H_{2}$
be the half space generated by $G\left( z\overline{y}\right) $ that does not
contain $x$. Let $D=A\cap H_{1}\cap H_{2}$. $D$ has a non-empty intersection
with the geodesic segment $\left( rs\right) $ i.e. $D$ has points close to $%
\overline{y}$. Since $D$ is compact, $D$ has an extreme point $p\in \partial
D$ by Lemma \ref{L2}. The boundary points of $D$ are either boundary points
of $A$ or points of $G\left( x\overline{y}\right) $. Thus $p$ is an extreme
point of $A$ i.e. $p$ sees $x$ via $A$. Since $z\in \ker A,$ $z$ sees the
geodesic segment $\left[ px\right] $ via $A$ and consequently $\left[ x%
\overline{y}\right] \subset A$ which is a contradiction and the point $x$
does not exist.
\end{proof}

\begin{theorem}
Let $W\in \mathcal{W}^{2}$ and let $A$ be a compact subset of $W$. Suppose
that $B=\bigcap\limits_{x\in E\left( A\right) }A_{x}\neq \varphi $. Then $%
\ker A=B$ if and only if for every $x\notin A$, there is a geodesic ray with
vertex at $x$ having a non-empty intersection with $A$.
\end{theorem}

\begin{proof}
Suppose that $\ker A\neq B$ i.e. $B\nsubseteq \ker A$. Let $y\in B\backslash
\ker A$. Thus there is a point $z\in A$ such that $\left[ yz\right]
\nsubseteq A$. Then by Lemma \ref{L1}, there are two pints $\overline{y}$
and $\overline{z}$ in $\partial A\cap \left[ yz\right] $ such that $\left( 
\overline{y}\overline{z}\right) \cap A=\phi $.\ Let $p\in \left( \overline{y}%
\overline{z}\right) $, then we get a point $\overline{p}\notin A$ such that
the geodesic ray $R\left( p\overline{p}\right) $ has a non-empty
intersection with $A$. Rotate the ray $R\left( p\overline{p}\right) $ to
touch $\partial A$ such that $p$ is fixed and the angle between $\left[ p%
\overline{p}\right] $ and $\left[ pz\right] $ decreases. The intersection of
the new geodesic ray and $A$ has an extreme point $x$ of $A$. Thus $y$ sees $%
x$ via $A$ and $\left[ xy\right] $ cuts the geodesic ray $R\left( p\overline{%
p}\right) $ in a point $a$ which is a contradiction otherwise $a\in \left[ yx%
\right] $ which is also a contradiction. Thus $\ker A=B.$

To prove the second implication, let $p\notin A$ and $q\in \ker A$. Consider
the geodesic ray $R\left( qp\right) $ passing through $p$. The geodesic ray $%
R\left( qp\right) \backslash \left[ qp\right) $ has a non-empty intersection
with $A$ otherwise $q\notin \ker A$.
\end{proof}

\begin{corollary}
Let $W\in \mathcal{W}^{2}$ and let $A$ be a compact subset of $W$. Then $A$
is starshaped if and only if $\bigcap\limits_{x\in E\left( A\right)
}A_{x}\neq \phi $ and for every $x\notin A$, there is a geodesic ray with
vertex at $x$ having a non-empty intersection with $A$. Moreover, $\ker
A=\bigcap\limits_{x\in E\left( A\right) }A_{x}$.
\end{corollary}

\begin{theorem}
\label{th a}Let $W\in \mathcal{W}^{n}$ and let $A$ be a non-empty closed
subset of $W$. If $\partial A$ is convex, then $A$ is a convex set.
Moreover, if $A$ has a non-empty interior, then $A$ is unbounded, $\partial
A $ is totally geodesic and $A^{c}$ is also convex.
\end{theorem}

\begin{proof}
Suppose that $A$ is not convex i.e. we get two points $p$ and $q$ in $A$
such that $\left( pq\right) $ is not contained in $A$. Since $A$ is closed,
we find two points $r,s$ in $\partial A$ such that $\left( rs\right) \cap
A=\phi $ which is a contradiction and so $A$ is convex.

To show that $A$ is unbounded, let $p\in int\left( A\right) $. Suppose that $%
A$ is bounded and so we find a real number $\varepsilon $ such that $A$ is
contained in the closed ball $\overline{B}\left( p,\varepsilon \right) $ of
radius $\varepsilon $ and center at $p$. Let $\left[ ab\right] $ be any
chord of $\overline{B}\left( p,\varepsilon \right) $ that runs through $p$.
Since $A$ and $\left[ ab\right] $ are both closed and convex sets, we find $%
a^{\prime }$ and $b^{\prime }$ in $\partial A$ such that $A\cap \left[ ab%
\right] =\left[ a^{\prime }b^{\prime }\right] $ which is a contradiction
since $\left[ a^{\prime }b^{\prime }\right] $ cuts the interior of $A$.
Therefore $A$ is unbounded.

Assume that $\partial A$ is not totally geodesic i.e. there are two points $%
a $ and $b$ in $\partial A$ such that the line $G\left( ab\right) $ passing
through $a$ and $b$ is not contained in $\partial A.$ Since $\partial A$ and 
$G\left( ab\right) $ are closed convex sets, there are $a^{\prime }$ and $%
b^{\prime }$ in $\partial A$ such that 
\begin{equation*}
\left[ ab\right] \subset \partial A\cap G\left( ab\right) =\left[ a^{\prime
}b^{\prime }\right]
\end{equation*}

Let $p\in G\left( ab\right) \backslash \left[ a^{\prime }b^{\prime }\right] $
(i.e. $p\in int\left( A\right) \cap G\left( ab\right) $ or $p\in A^{c}\cap
G\left( ab\right) $) and assume that $p\in R\left( b^{\prime }a^{\prime
}\right) $. If $p\in int(A)\cap G\left( ab\right) $, then the geodesic
convex cone $C\left( b,\overline{B}\left( p,\varepsilon \right) \right) $
with vertex $b$ and base $\overline{B}\left( p,\varepsilon \right) $ for a
sufficiently small $\varepsilon $ shows \ that $a$ is an interior point
which is a contradiction see Figure \ref{07-01}.

\begin{figure}[h]
\begin{center}
\includegraphics{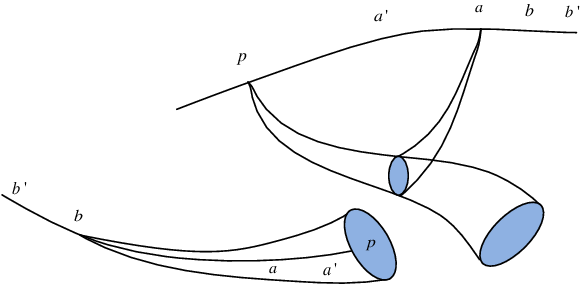}
\end{center}
\caption{Two cases for the point $p$}
\label{07-01}
\end{figure}

Now we take $p\in A^{c}\cap G\left( ab\right) $. Let $q$ be a point of $%
int\left( A\right) $. The sets $\left[ pq\right] $ and $A$ are closed convex
sets and so there is a point $q^{\prime }\in \partial A$ such that $\left[ pq%
\right] \cap A=\left[ q^{\prime }q\right] $. This implies that the
intersection $B=\partial A\cap C\left( p,\overline{B}\left( q,\epsilon
\right) \right) $, for a small $\epsilon $, is a non-empty closed convex set
since $\partial A$ is convex. Therefore, $B$ is a convex cross section of $%
C\left( p,\overline{B}\left( q,\epsilon \right) \right) $ that determines a
hypersurface $H$ whose intersection with $C\left( p,\overline{B}\left(
q,\epsilon \right) \right) $ is $B$. At least one of the points $a$ and $b$%
(say $a$) does not lie in $H$ otherwise the line $G\left( ab\right) $ lies
in $H$ which contradicts the fact that $p$ is the vertex of the convex cone $%
C\left( p,\overline{B}\left( q,\epsilon \right) \right) $. Now, the convex
cone $C\left( a,B\right) $ has dimension $n$ i.e. $C\left( a,B\right) $ has
interior points which is a contradiction since both $a$ and $B$ are in $%
\partial A$ see Figure \ref{07-01}. This contradiction completes the proof.
\end{proof}

\begin{corollary}
Let $W\in \mathcal{W}^{n}$ and let $A$ be a non-empty open subset of $W$ and 
$int(\overline{A})=A$. If $\partial A$ is convex, then $A$ is unbounded
convex set and $\partial A$ is totally geodesic.
\end{corollary}

\begin{proof}
It is clear that $\overline{A}$ satisfies the hypothesis of Theorem \ref{th
a}. Therefore $\partial \overline{A}=\partial A$ is affine and $\overline{A}$
is convex and unbounded. Note that if $A$ is bounded, then $\overline{A}$ is
also bounded and equivalently, $\overline{A}$ is unbounded implies that $A$
is unbounded. Since the interior of a closed convex set is also convex, the
convexity of $\overline{A}$ implies that $A$ is convex.
\end{proof}

\begin{theorem}
Let $W\in \mathcal{W}^{n}$ and let $A$ be a non-empty closed subset of $W.$
If $\left( pq\right) \subset int\left( A\right) $ for every pair of boundary
points $p,q$ of $A$, then $A$ is strictly convex.
\end{theorem}

\begin{proof}
It is enough to prove that $A$ is convex since the strict convexity of $A$
is direct. Now, we assume that $A$ is not convex i.e. there are $p,q$ in $A$
such that $\left[ pq\right] $ is not contained in $A.$ Since $A$ is closed,
there are $p^{\prime },q^{\prime }$ in $\partial A$ such that $\left(
p^{\prime }q^{\prime }\right) \cap A=\phi $ which is a contradiction and $A$
is convex.
\end{proof}

\begin{corollary}
Let $W\in \mathcal{W}^{n}$ and let $A$ be a non-empty closed subset of $W$. $%
A$ is convex if and only if $\left( pq\right) \subset \partial A$ or $\left(
pq\right) \subset int\left( A\right) $ for each pair of boundary points $p,q$%
.
\end{corollary}

Since the interior of a closed convex set is again convex, this result is
still true for open sets such that $int\left( \overline{A}\right) =A$. The
following example shows that the closeness is important. Let $A$ be a subset
of $E^{2}\in \mathcal{W}^{2}$ defined by $A=\left\{ \left( x,y\right)
:0\prec x\prec 1,0\prec y\prec 1\right\} \cup \left\{ \left( 0,0\right)
,\left( 1,1\right) ,\left( 1,0\right) ,\left( 0,1\right) \right\} .$ $A$ is
neither closed nor open and $\left( pq\right) \subset \partial A$ or $\left(
pq\right) \subset int\left( A\right) $ for each pair of boundary points $p,q$
but $A$ is not convex.

\begin{proposition}
Let $W\in \mathcal{W}^{n}$ and let $A$ be a non-empty closed subset of $W$.
If there is a point $p\in A$ that sees $\partial A$ via $A$, then $A$ is
starshaped.
\end{proposition}

\begin{proof}
We claim that $p\in \ker A$. Suppose that $p$ is not in $\ker A$ i.e. there
is a point $q\in A$ such that $\left[ pq\right] $ is not contained in $A$.
Since $A$ is closed, there are two points $p^{\prime}$ and $q^{\prime}$ in $%
\partial A\cap \left[ pq\right] $ such that $\left(
p^{\prime}q^{\prime}\right)\cap A=\phi$. Thus $p$ does not see neither $%
p^{\prime}$ nor $q^{\prime}$. This contradicts the fact that $p$ sees $%
\partial A$ via $A$ and the proof is complete.
\end{proof}

It is clear that the converse of this result is also true. Thus we can say
that this proposition is a characterization of the kernel of the closed
starshaped sets. This means that the kernel of a closed starshaped set $A$
is only the points of $A$ that see $\partial A$. The following corollary is
direct.

\begin{corollary}
Let $W\in \mathcal{W}^{n}$ and let $A$ be a non-empty closed convex subset
of $W$. If $\partial A$ is starshaped, then $\ker \left( \partial A\right)
\subset \ker A.$
\end{corollary}

In the light of the above results, one can test the convexity and
starshapedness of a closed set $A$ using its boundary points. In the next
part a minimal subset of these boundary points will build $A$ up from inside.

\begin{theorem}
\label{th b}Let $W\in \mathcal{W}^{n}$ and let $A$ be a non-empty closed
convex subset of $W$. If $A$ has no hyperplane, then $A=C\left( \partial
A\right) $.
\end{theorem}

\begin{proof}
Since $A$ is convex, $A$ is connected. We will prove that $C\left( \partial
A\right) $ is open and closed in the relative topology on $A$ and hence $%
A=C\left( \partial A\right)$.

First, we prove that $C\left( \partial A\right) $ is open in $A$. Let $p\in
C\left( \partial A\right) \subset A$. We have the following cases:

\begin{enumerate}
\item $p\in C\left( \partial A\right) \cap int\left( A\right) $: Let $%
B_{\delta}=B\left(p,\delta \right)\cap A$. In this case there exists a real
number $\delta $ such that $B\left( p,\delta \right) \subset A$ and so $%
B_{\delta}=B\left(p,\delta \right)$. Suppose that $p$ is not an interior
point of $C(\partial A)$ i.e. for any $\delta $, $B_{\delta}$ is not
contained in $C\left( \partial A\right)$ and so $p$ is a boundary point of $%
C(\partial A)$. Therefore, there is a supporting hyperplane $H_{1}$ of $%
\overline{C(\partial A)}$(the closure of $C(\partial A)$ is a closed convex
subset of $A$) at $p$ and $\overline{C(\partial A)}$ is contained in a
closed half-space with boundary $H_{1}$. Let $x$ be any point of $B\left(
p,\delta \right) $ that lies on the other side of $H_{1}$ and let $H_{2}$ be
a parallel hyperplane to $H_{1}$ at $x$. Since $A$ does not contain a
hyperplane, we find a point $y\in H_{2}\setminus A$. The line segment $\left[
xy\right] $ cuts $\partial A$ at a point $z\in H_{2}$ which contradicts the
fact that $H_{1}$ supports $\overline{C\left(\partial A\right)}$. This
contradiction implies that $p$ is an interior point of $C\left( \partial
A\right) $ in the relative topology of $A$.

\item $p\in C\left( \partial A\right) \cap \partial A$: in this case, $%
B\left( p,\delta \right) $ has a non-empty intersection with $A$ for any
real number $\delta $. Let $B_{\delta }=B\left( p,\delta \right)\cap A$.
Suppose that $p$ is not an interior point of $C\left( \partial A\right)$.
Then, for any $\delta$, the set $B_{\delta }$ has a point $x$ which is not
in $C\left( \partial A\right)$. But $\overline{C\left( \partial A\right)} $
is closed convex set and $x\notin \overline{C\left( \partial A\right)}$, and
so we get a hyperplane $H$ passing through $x$ that separates $x$ and $%
\overline{C\left( \partial A\right)}$. Since $A$ does not have a hyperplane,
there is a point $y$ in $H\setminus A$ where $[xy]$ cuts $\partial A$. Thus $%
H$ cuts $\partial A$ and so $H$ cuts $C\left( \partial A\right) $ which is a
contradiction and so $p$ is an interior point of $C\left( \partial A\right) $
in the relative topology on $A$.
\end{enumerate}

This discussion above implies that $C\left( \partial A\right) $ is an open
set in $A$. Now, we want to prove that $C\left( \partial A\right) $ is
closed in $A$. Let $p$ be a boundary point of $C\left( \partial A\right)$.
If $p\in \partial A,$ then $p\in C\left( \partial A\right) .$ Let $p\in intA 
$, then there is a small positive real number $\delta $ such that $B\left(
p,\delta \right) \subset A.$ Since $p$ is a boundary point of $C\left(
\partial A\right) $, $B\left( p,\delta \right) \neq B\left( p,\delta \right)
\cap C\left( \partial A\right) \neq \phi $. Therefore, we find a point $x$
in $B\left( p,\delta \right) $ which is not in $C\left( \partial A\right) $.
Since $\overline{C\left( \partial A\right)} $ is a closed convex set, we get
a hyperplane $H$ passing through $x$ and does not intersect $C\left(
\partial A\right) $. But $A$ does not have a hyperplane and so $H$ cuts $%
\partial A$ which is a contradiction and $p\in C\left( \partial A\right) $
i.e. $C\left( \partial A\right) $ is closed in the relative topology on $A$
and the proof is complete.
\end{proof}

In general, sets need not have extreme points. The following proposition
gives a sufficient condition for the existence of extreme points.

\begin{proposition}
Let $W\in \mathcal{W}^{n}$ and let $A$ be a non-empty closed convex subset
of $W$. $A$ contains at least one extreme point if and only if $A$ has no
geodesic.
\end{proposition}

\begin{proof}
Let us assume that $A$ has a geodesic $l$. Suppose that $A$ has an extreme
point $p$. It is clear that $p\notin l$. Let $B$ be the closed convex hull
of $p$ and $l$. $B$ is a subset of $A$ since $A$ is a closed convex set
containing both $p$ and $l$. It is clear that $B$ contains a line passing
through $p$ and parallel to $l$ i.e. either $p$ is not an extreme point or
the line $l$ does not exist.
\end{proof}

\begin{lemma}
Let $W\in \mathcal{W}^{n}$ and let $A$ be a non-empty closed convex subset
of $W$. If $H$ is a supporting totally geodesic hypersurface of $A$, then $%
E\left( H\cap A\right) \subset E\left( A\right) $
\end{lemma}

\begin{proof}
Let $p$ be an extreme point of $H\cap A$. Suppose that $p\notin E\left(
A\right) $ i.e. there are $x,y$ in $\partial A$ such that $p\in \left(
xy\right) .$ The hypersurface $H$ supports $A$ at $p$ and so $\left[ xy%
\right] \subset H$. This implies that $p\in \left[ xy\right] \subset H\cap A$
which contradicts the fact that $p$ is an extreme point of $H\cap A$. This
contradiction completes the proof.
\end{proof}

The minimal subset of a compact convex set $A$ which generates $A$ is is its
extreme points. Our next main theorem shows that this property is more
general.

\begin{theorem}
\label{Th c}Let $W\in \mathcal{W}^{n}$ and let $A$ be a non-empty closed
convex subset of $W$. If $A$ has no hyperplane and its boundary has no ray,
then $A=C\left( E\left( A\right) \right) $.
\end{theorem}

\begin{proof}
To prove that $A=C\left( E\left( A\right) \right) $, it suffices to prove
that $\partial A\subset C\left( E\left( A\right) \right) $ and by Theorem %
\ref{th b}, we get that $A=C\left( \partial A\right) \subset C\left( E\left(
A\right) \right) \subset A$ and hence $A=C\left( E\left( A\right) \right) $.
Let $p\in \partial A$. If $p$ is an extreme point, then $p\in E\left(
A\right) \subset C\left( E\left( A\right) \right) $. Now suppose that $p$ is
not an extreme point i.e. there are $x,y$ in $\partial A$ such that $p\in
\left( xy\right) $. Since $A$ is a closed convex set, there is a supporting
totally geodesic hypersurface $H$ of $A$ at $p$. It is clear that the set $%
H\cap A$ is a non-empty closed convex subset of $\partial A$. Since $%
\partial A$ has no ray, the set $H\cap A$ is bounded i.e. $H\cap A$ is a
compact convex set. Therefore, $H\cap A=C\left( E\left( H\cap A\right)
\right) $. But $E\left( H\cap A\right) \subset E\left( A\right) $ and so $%
p\in C\left( E\left( H\cap A\right) \right) \subset C\left( E\left( A\right)
\right) $ and the proof is complete.
\end{proof}

\end{document}